\documentclass[12pt]{amsart}
\usepackage{a4wide, amssymb,latexsym, amscd}
\usepackage[all]{xy}
\usepackage{graphicx}
\usepackage{mathrsfs}
\usepackage{amsmath}
\usepackage{pb-diagram}

 \newtheorem{theorem}{Theorem}[section]
 \newtheorem{cor}[theorem]{Corollary}
\newtheorem{conjecture}[theorem]{Conjecture}
 \newtheorem{lemma}[theorem]{Lemma}
 \newtheorem{proposition}[theorem]{Proposition} \theoremstyle{definition}
 
 \theoremstyle{definition}
 
 \theoremstyle{remark}
 \newtheorem{remark}[theorem]{Remark}
 \numberwithin{equation}{section}

\numberwithin{equation}{section}

\newcommand{\End}{\operatorname{End}}

\newcommand{\Aut}{\operatorname{Aut}}
\newcommand{\AL}{\operatorname{AL}}

\newcommand{\gr}{\operatorname{gr}}

\newcommand{\ii}{\hbox{\ensuremath{\operatorname{i}}}}

\newcommand{\mSpec}{\ensuremath{\operatorname{max}}}

\newcommand{\F}{\ensuremath{\mathbf{k}}}
\newcommand{\N}{\ensuremath{\mathbb{N}}}
\newcommand{\Z}{\ensuremath{\mathbb{Z}}}
\newcommand{\w}{\ensuremath{\widetilde}}
\newcommand{\Q}{\ensuremath{\mathbb{Q}}}

\newcommand{\C}{\ensuremath{\mathbb{C}}}

\newcommand{\x}{\ensuremath{\mathbf{x}}}
\newcommand{\bpa}{\ensuremath{\mathbf{d}}}

\newcommand{\Atm}{\ensuremath{A^{(t)}_m}}

\newcommand{\Atn}{\ensuremath{A^{(t)}_n}}
\newcommand{\Aqn}{\ensuremath{A^{q}_n}}
\newcommand{\At}{\ensuremath{A^{(t)}}}
\newcommand{\Aq}{\ensuremath{A^{q}}}
\newcommand{\Aqnprime}{\ensuremath{A^{q'}_n}}

\newcommand{\Atnf}{\ensuremath{A^{(t)}_{n,\f}}}
\newcommand{\Atf}{\ensuremath{A^{(t)}_{\f}}}

\newcommand{\Aqnf}{\ensuremath{A^{q}_{n,\f_q}}}

\newcommand{\Aqnfprime}{\ensuremath{A^{q'}_{n,\f_{q'}}}}

\newcommand{\Zqn}{\ensuremath{Z^{q}_n}}

\newcommand{\Zq}{\ensuremath{Z^{q}}}

\newcommand{\Zqnf}{\ensuremath{Z^{q}_{n,\f^l_q}}}

\newcommand{\Jm}{\operatorname{Im}}

\newcommand{\Id}{\operatorname{Id}}

\newcommand{\ot}{\operatorname{\otimes}}
\newcommand{\f}{\ensuremath{\operatorname{f}}}
\newcommand{\fq}{\ensuremath{\operatorname{f}_q}}
\newcommand{\flq}{\ensuremath{\operatorname{f}^l_q}}

\newcommand{\m}{\ensuremath{\mathfrak{m}}}

\newcommand{\BGG}{\ensuremath{\mathcal{O}}}

\newcommand{\U}{\mathbb{U}}
\newcommand{\Ul}{\mathbb{U}_l}
\newcommand{\Ulp}{\mathbb{U}^{prim}_l}

\newcommand{\la}{\ensuremath{\langle}}
\newcommand{\ra}{\ensuremath{\rangle}}

\newcommand{\pa}{\ensuremath{\partial}}

\author{Erik Backelin}
\title[Endomorphisms of quantized Weyl algebras]{Endomorphisms of quantized Weyl algebras}
\address{Departamento de Matem\'{a}ticas, Universidad de los Andes,
Carrera 1 N. 18A - 10, Bogot\'a, COLOMBIA}
\email{erbackel@uniandes.edu.co}
\begin{document}
\maketitle
\begin{abstract}
Belov-Kanel and Kontsevich, \cite{B-KK2}, conjectured that the
group of automorphisms of the $n$'th Weyl algebra and the group of
polynomial symplectomorphisms of $\C^{2n}$ are canonically
isomorphic. We discuss how this conjecture can be approached by
means of (second) quantized Weyl algebras at roots of unity.
\end{abstract}

\section{Introduction}
We recall Dixmier's conjecture and Belov-Kanel and Kontsevich's
conjecture about the Weyl algebra $A_n$ and discuss how the latter
can be approach by means of quantized Weyl algebras. We also
discuss endomorphisms verses automorphisms of quantized Weyl
algebras.
\subsection{Conjectures of Dixmier and of Belov-Kanel and Kontsevich} Let $\Lambda$ be a commutative unitary ring. The $n$'th Weyl
algebra $ A_n(\Lambda)$ ($n > 0$) is defined to be
$$\Lambda \la x_1, \ldots , x_n, \pa_1, \ldots , \pa_n \ra/I$$
where $I$ is the ideal generated by $\pa_i x_j - x_j \pa_i -
\delta_{ij}$, $x_ix_j- x_jx_i$ and $\pa_i \pa_j - \pa_j \pa_i$ for
$1 \leq i,j \leq n$. We put $A_n = A_n(\C)$.

\smallskip

\noindent  Let us recall some basic conjectures. Let $\Aut(A_n)$
be the group of algebra automorphisms of $A_n$ and $\End(A_n)$ be
the monoid of algebra endomorphisms of $A_n$. \textbf{Dixmier's
conjecture} $\operatorname{D}_n$ states that $\Aut(A_n) =
\End(A_n)$, (see \cite{D}).

The \textbf{Jacobian conjecture} $\operatorname{J}_n$ states that
any \'etale map $\C^n \to \C^n$ is a homeomorphism.

The conjecture $\operatorname{D}_n$ is open for all $n$. It is
known that $\operatorname{D}_n \implies \operatorname{J}_n$ and
that $\operatorname{J}_{2n} \implies \operatorname{D}_n$; this was
proved in \cite{ACvE2}, and independently in \cite{B-KK}.

Let $P_n (\Lambda)$ be the polynomial ring $\BGG_{{2n}}(\Lambda)
:= \Lambda[r_1, \ldots , r_n, s_1, \ldots s_n]$ together with its
canonical Poisson bracket $\{ \ , \ \}$, where $\{r_i, s_j\} =
\delta_{ij}$ and $\{r_i,r_j\} = \{s_i, s_j\} = 0$. Thus
$P_n(\Lambda)$ is the associated graded structure of
$A_n(\Lambda)$ with respect to the order filtration where $\deg
x_i = 0$ and $\deg \pa_i = 1$. Put $\BGG_{2n} = \BGG_{2n}(\C)$ and
$P_n = P_n(\C)$.

Denote by $\Aut(P_n)$ the automorphism group of $P_n$ (i.e., the
group of algebra automorphisms of $\BGG_{2n}$ that preserves  $\{
\ , \ \}$). Similarly, we have the monoid $\End(P_n)$ and
conjecturally $\End(P_n) = \Aut(P_n)$ and this is equivalent to
$\operatorname{J}_n$, see \cite{ACvE2}.

 \textbf{Belov-Kanel and Kontsevich}, \cite{B-KK2},
conjectured the $\operatorname{B-KK}_n$ \textbf{conjecture}: The
groups $\Aut(A_n)$ and $\Aut(P_n)$ are canonically isomorphic.

\smallskip

In fact, they conjectured that this holds also with $\C$ replaced
by $\Q$. The structure of these groups is well-known for $n=1$ and
we shall recall it in section \ref{known structure groups}. The
conjecture $\operatorname{B-KK}_n$ is proved only for $n = 1$.

\smallskip

\noindent Let us say a few words about the known proofs of the
equivalence between the Jacobian and Dixmier's conjectures,
\cite{ACvE2}, \cite{B-KK},  and about the approach of Belov-Kanel
and Kontsevich to their own conjecture, \cite{B-KK2}.

The basic observation is that the Weyl algebra over a ring in
prime characteristic is Azumaya. One then has the following
general results established in \cite{ACvE2} and \cite{B-KK}:

\smallskip

\noindent \textbf{Theorem A)} \emph{An endomorphism of an Azumaya
algebra maps its center into itself} and

\smallskip

\noindent \textbf{Theorem B)} \emph{the endomorphism is an
automorphism iff its restriction to the center is an
automorphism.}

With this in hand, in order to prove that $\operatorname{J}_{2n}$
implies $\operatorname{D}_n$ one starts with an endomorphism
$\phi$ of $A_n$. $\phi$ will induce an endomorphism of $A_n(S)$
where $S$ is a finitely generated subring of $\C$. Taking
reduction modulo primes one gets that this induces an endomorphism
of the center of $A_n(S/p)$, which is a polynomial ring in the
$2n$ commuting variables $x^p_i, \pa^p_i$, $1 \leq i \leq n$, for
each prime $p$. By $\operatorname{J}_{2n}$ and Lefschetz's
principle (or L$\hat{\hbox{o}}$s' theorem in logic) it is an
automorphism for all sufficiently large $p$. Then
$\operatorname{D}_n$ follows again from Lefschetz's principle.

The $\operatorname{B-KK}_n$-conjecture is more delicate. Given an
endomorphism $\phi$ of $A_n$, the reduction modulo prime strategy
leads to an endomorphism $$\overline{\phi}: \BGG_{2n}(\C_{\infty})
\to \BGG_{2n}(\C_{\infty}),$$ that preserves $\{\ , \ \}$, where
$\C_{\infty}$ is a ring extension of $\C$ which is an ultra
product of rings of prime characteristic. Conjecturally the map
$\overline{\phi}$ take values in $\BGG_{2n}$ and after untwisting
$\overline{\phi}$ with the non-standard Frobenius one
conjecturally obtain the desired isomorphism.

The conjectured isomorphism $\Aut(A_n) \cong \Aut(P_n)$ must be
something quite complicated; the involved groups are $\C$-points
of certain ind-group schemes which are known to be non-isomorphic,
see \cite{B-KK2} for details.

\subsection{Quantized Weyl algebras and the
$\operatorname{B-KK}_n$-conjecture.}\label{intro quant} In this
note we propose that instead of using Weyl algebras in prime
characteristic one may use quantized Weyl algebras at roots of
unity. We discuss endomorphisms and automorphisms of these
algebras. Then we approach the $\operatorname{B-KK}_n$-conjecture
by letting the root of unity parameter converge to $1$.

This will also lead to difficult convergency problems but has the
advantage over the reduction modulo prime method that we never
leave the field $\C$; on the other hand, the fact that (standard)
Frobenius maps are not available on $\C$ leads to disadvantages.
Also, our work here is rather speculative and probably incomplete.
Our method can be thought of as another attempt to give a concrete
meaning to the philosophical statements $[\pa^\infty, x^\infty] =
0$ and $\{\pa^\infty , x^\infty \} = 1$ which we think deserves to
be compared to the reduction modulo prime method.

Let $t$ be a parameter and let $R = \C[t,t^{-1}]$ be our base
ring. Let $n
> 0$ and define the $n$'th quantized Weyl algebra
\begin{equation}\label{defQWAintro}
\Atn = R\la x_1, \ldots , x_n, \pa_1, \ldots , \pa_n \ra/I_t
\end{equation}
where $I_t$ is the ideal generated by $\pa_i x_j - t^{\delta_{ij}}
x_j \pa_i - \delta_{ij}$,  $x_ix_j- x_jx_i$ and $\pa_i \pa_j -
\pa_j \pa_i$ for $1 \leq i,j \leq n$. Let $\Aqn = \Atn/(t-q)$ be
the specialization of $\Atn$ to $q \in \C^*$. Let $ \pi_q: \Atn
\to \Aqn, \ g \mapsto g_q := \pi_q(g)$ be the specialization map.

We shall see in proposition \ref{naively false D} that Dixmier's
conjecture is false for $\Atn$ and for $\Aqn$ for any $n$ and $q
\neq 1$.

If $q$ is a primitive $l$'th root of unity ($l
>1$) the center $\Zqn$ of $\Aqn$ is a polynomial ring in
$2n$ variables
$$
\Zqn = \C[x^l_1, \ldots, x^l_n, \pa^l_1, \ldots , \pa^l_n].
$$
If $q$ is not a root of unity $\Zqn = \C$.

Let $A$ be a $\C$-algebra which we assume is free of rank $N^2$
over its center $Z$, for some $N \in \N$. We assume that $Z$ is a
finitely generated $\C$-algebra. Recall that the Azumaya locus of
$A$ is defined to be
$$
\AL(A) = \{\m \in \mSpec(Z); \ A/(\m) \cong M_N(\C)\}
$$
$A$ is called Azumaya if $\AL(A) =  \mSpec(Z)$ (see \cite{DI},
\cite{M} for generalities of Azumaya algebras).

Let us say that $A$ is \emph{almost} Azumaya if $\AL(A)$ is
Zariski dense in $\mSpec(Z)$. Note that $\AL(A)$ automatically is
Zariski open in $\mSpec(Z)$.

We prove that $\Aqn$ is almost Azumaya of rank $l^{2n}$ over its
center $\Zqn$, when $q$ is a primitive $l$'th root of unity ($l >
1$). This is done by explicitly computing $\AL(\Aqn)$ (proposition
\ref{AL}).

A main result of this paper is theorem \ref{almost preserve
center}. It shows that a generalization of theorem \textbf{A)}
above holds for $\Aqn$:  given $\phi \in \End(\Atn)$ the
specialization $\phi_q := \pi_q(\phi) \in \End(\Aqn)$ will map the
center $\Zqn$ into itself for almost all roots of unity $q$. A
crucial ingredient in the proof of this theorem is a distinguished
element
$$
\f := \prod^n_{i=1} [\pa_i, x_i]  \in \Atn.
$$
There is the Ore-localization $\Atnf$ obtained by inverting $\f$.
Then $\Aqnf := \Atnf/(t-q)$ is Azumaya for any root of unity $q$
because the complement of $\AL(\Aqn)$ in $\mSpec(\Zqn)$ is
precisely the zero set of the polynomial $\f^l_q \in \Zqn$ (see
lemma \ref{seclem}).

We compute $\Aut(\Atnf)$ in proposition \ref{endomorphisms of A_t}
and show that $\End(\Atnf) = \Aut(\Atnf)$. This is a nice result,
but let us remark that it unfortunately cannot be directly used to
attack $\operatorname{D}_n$ because endomorphisms of $A_n$ will in
general not lift to endomorphisms of $\Atnf$. On the other hand we
expect (conjecture \ref{ConjA}) that any element of $\End(A_n)$
has a lift in $\End(\Atn)$. Again, it is not clear what
implications an affirmative answer to this conjecture could have
on Dixmier's conjecture, because $\Atn$ does have non-invertible
endomorphisms. However, we shall explicitly show that all
automorphisms of $A_1$ admit lifts to $\End(A^{(t)}_1)$, hence a
negative answer to conjecture \ref{ConjA} for $n=1$ would imply
that $\operatorname{D}_1$ fails as well.

\smallskip

\noindent Let us finally explain how all this can be used to
approach the $\operatorname{B-KK}_n$ conjecture. It goes in two
steps:
$$
\textbf{1)} \ \Aut(A_n) \to \End(\Atn)' \hbox{ and } \textbf{2)} \
\;\widehat{ \ }\;: \End(\Atn)' \to \Aut(P_n)
$$
where $\End(\Atn)'$ is a certain submonoid of $\End(\Atn)$.%

Step \textbf{1)} is probably the most difficult and at present not
clear how it should be properly formulated. Given $\phi \in
\Aut(A_1)$ we conjecture that there is a lift $\w \phi \in
\End(\Atn)'$, i.e., $({\w \phi})_1 = \phi$, see conjecture
\ref{ConjA'}. Again, the evidence we have are computations for
$n=1$. The lift is (if it exists) non-unique.

For step \textbf{2)} let $\phi \in \End(\Atn)$ be given. We know
that for almost all roots of unit $q$, we get an endomorphism
$\phi_q |_{\Zqn}$ of $\Zqn$. Let $q_l := \exp(2\pi \ii/l)$, where
$\ii \in \C$ is an imaginary unit, and let $\Theta_l: \BGG_{2n}
\to \Zqn$ be the algebra isomorphism defined by $\Theta_l(r_i) =
\pa^l_i$ and $\Theta_l(s_i) = x^l_i$, for $i = 1, \ldots , n$.
Define
$$
\widehat{\phi}(P) := \lim_{l \to \infty} \Theta^{-1}_l \circ \phi
\circ \Theta_l(P)
$$
for $P \in \BGG_{2n}$. We say that $\phi$ \emph{converges} if the
right hand side of this formula converges to a polynomial in
$\BGG_{2n}$ for all $P \in \BGG_{2n}$ and we put $\End(\Atn)' =
\{\phi \in \End(\Atn); \phi \hbox{ converges }\}$. Apriori, in the
case that $\phi$ is convergent, $\widehat{\phi}$ is in $\End(P_n)$
and conjecturally $\End(P_n) = \Aut(P_n)$; This gives our map
$\;\widehat{ \ }\;$ of \textbf{2)} which turns out to be a
morphism of monoids.

We shall see that the canonical Poisson bracket $\{ \ , \ \}$ on
$\BGG_{2n}$ is the limit of certain Poisson brackets on $\Zqn$
which are defined in terms of the Lie bracket on $\Aqn$; this
implies that the bracket is preserved by $\widehat{\phi}$. We
expect that the image of $\; \widehat{ \ }\;$ is a large subgroup
of $\Aut(P_n(\Z))$, see conjecture \ref{hat is surjective}. The
evidence we have for this comes from considering the case $n = 1$,
essentially the computation in proposition \ref{evidence at 1}.
(If the map $a \mapsto a^l$ on $\C$ would have been a field
automorphism we could have given a better definition of
$\;\widehat{ \ }\;$, see remark \ref{Frobenius on C}. In any case,
it is plausible that we are note working with the correct notion
of limits.)

A problem with \textbf{1)} is that given $\phi \in \Aut(A_n)$,
$\widehat{\widetilde{\phi}}$ will depend on the choice of $\w
\phi$, see corollary \ref{different lifts same result} for an
example. We have so far not been able to resolve this problem,
e.g. by characterizing ``good" lifts. Ideally, there should be a
well-defined map $\;\w { \ }\; : \Aut(A_1) \to \End(\Atn)'$ such
that $\;\widehat { \ }\, \circ \,\w { \ }\; |_{\Aut(A_1(\Z))}$ is
a group homomorphism.

\section{Quantized Weyl algebras at roots of unity}
In section \ref{Quantized Weyl algebra} we give basic properties
about the quantized Weyl algebra $\Atn$ and its endomorphisms. In
section \ref{Azumaya property at a root of unity.} we proceed to
describe the center and the Azumaya locus of $\Aqn$ for $q$ a root
of unity. We then construct in section \ref{section Azumaya loc} a
localization $\Atnf$ of $\Atn$ whose specialization to any root of
unity is Azumaya and use this in section \ref{Endomorphisms of Aqn
preserve the center} to prove that an endomorphism of $\Atn$
preserves the center of $\Aqn$ for almost all roots of unity $q$.
\subsection{Quantized Weyl algebra}\label{Quantized Weyl algebra}
In section \ref{intro quant} we defined the quantized Weyl algebra
$\Atn$ over the base ring $R = \C[t,t^{-1}]$, its specialization
$\Aqn = \Atn/(t-q)$ to $q \in \C^*$, the specialization map
$\pi_q: \Atn \to \Aqn, \ g \mapsto g_q := \pi_q(g)$.

We have the usual Lie bracket $[f,g] := fg-gf$ for $f,g$ elements
of an associative ring. We shall also use the notations
$$
[f,g]_t = fg-tgf, \hbox{ for } f,g \in \Atn \hbox{ and } [f,g]_q =
fg-qgf, \hbox{ for } f,g \in \Aqn
$$
for $q \in \C^*$.

When $n = 1$ we write $x = x_1$, $\pa = \pa_1$, $\At = \At_1$ and
$\Aq = \Aq_1$.

We use the quantum notations
$$[m]_t = (1-t^{m})/(1-t), \ [m]_t! = [m]_t \cdot
[m-1]_t \cdots [1]_t,$$ for $m \in \N$.  Then we see that in
$\Atn$ hold the relations
\begin{equation}\label{xcxw}
\pa_i \cdot x^m_i = [m]_{t} \cdot x^{m-1}_i + t^mx^m_i \pa_i, \;
\pa^m_i \cdot x_i = [m]_{t} \cdot \pa^{m-1}_i + t^mx_i \pa^m_i
\end{equation}
for $m \geq 0$, $1 \leq i \leq n$. Hence we can define an action
of $\At$ on the polynomial ring $R[x_1, \ldots , x_n]$ by letting
$x_i$ act by multiplication and
\begin{equation}\label{xcx}
\pa_i(x^m_i) = [m]_{t} \cdot x^{m-1}_i, \; \pa_j(x^m_i) = 0,
\hbox{ for }, \ 1 \leq i \neq j \leq n
\end{equation}
and we see that $R[x_1, \ldots , x_n] \cong \Atn/\Atn \cdot
(\pa_1, \ldots , \pa_n)$ as a left $\Atn$-module. The action
\ref{xcx} is faithful and from this fact one easily deduces that
\begin{equation}\label{ewertg}
\{\x^\alpha  \bpa^\beta;\, \alpha, \beta \in \N^n\}
\end{equation}
forms an $R$-basis of $\Atn$, where $\x^\alpha = x^{\alpha_1}_1
\dots x^{\alpha_n}_n$ and  $\bpa^\beta = \pa^{\beta_1}_1 \dots
\pa^{\beta_n}_n$. We shall refer to this as the PBW-basis of
$\Atn$.

It follows that \ref{ewertg} also defines a $\C$-basis for $\Aqn$
for each $q \in \C^*$. From this one deduces easily that $\Atn$
and $\Aqn$ are integral domains.

There is a positive increasing filtration on $\Atn$ (resp., on
$\Aqn$) such that $\deg \pa_i = \deg x_i = 1$ which is called the
Bernstein filtration. By the existence of the PBW-basis we see
that the associated graded ring is isomorphic to a ring of
$t$-skew-commutative (resp., $q$-skew-commutative) polynomials
over $R$ (resp., over $\C$) in $2n$ variables living in degree
$1$.

In the introduction we defined a distinguished element
$$\f = \prod^n_{i=1} \f_i \in \Atn,  \hbox{ where } \f_i =  [\pa_i, x_i] = 1-(1-t) x_i\pa_i.$$
We have
\begin{equation}\label{frel}
\pa_i \f_j = t^{\delta_{ij}} \f_j \pa_i \hbox{ and } \f_j x_i =
t^{\delta_{ij}} x_i \f_j, \; 1 \leq i,j \leq n
\end{equation}
From this it follows that given any $P = \sum a_{\alpha, \beta}
\x^\alpha \bpa^\beta \in \Atn$, if we put $Q = \sum a_{\alpha,
\beta} t^{\alpha_i - \beta_i}\x^\alpha \bpa^\beta$, then we
get
\begin{equation}\label{frelgeneral}
\f_i P = Q \f_i
\end{equation}
from which it readily follows that $\{ \f^n; n \geq 0\}$ is an
Ore-set in $\Atn$. We shall study the corresponding
Ore-localization $\Atnf$ in section \ref{section Azumaya loc} and
onward.

It is well-known that the Weyl algebra $A_n$ is simple. On the
other hand the algebras $\Atn$ and $\Aqn$ for $q \in \C^*$, $q
\neq 1$, are not simple. To see this, observe for instance first
that two-sided ideal $(\f) \subset \Atn$ (resp. $(\fq) \subset
\Aqn$, for $q \neq 1$, where $\f_q \overset{def}{=} \pi_q(\f)$) is
proper. Moreover, while every representation of $A_n$ is
necessarily infinite dimensional, we see that for any $q \neq 1$,
$\Aqn$ has one dimensional representations constructed as follows:
Given $\mathbf{a} = (a_1, \ldots , a_n) \in (\C^*)^n$ define a
representation $\C_{\mathbf{a}} = \C \cdot \mu_{\mathbf{a}}$ by
$$
x_i  \mu_{\mathbf{a}} := a_i  \mu_{\mathbf{a}}, \ \pa_i
\mu_{\mathbf{a}} := (1-q)^{-1}a^{-1}_i \mu_{\mathbf{a}}, \ 1 \leq
i \leq n
$$
Hence, at least naively formulated, Dixmier's conjecture is false
for quantized Weyl algebra:
\begin{proposition}\label{naively false D} Let $q \in \C^* \setminus \{1 \}$. \textbf{i)} $\Atn$ and
$\Aqn$ have endomorphisms which are injective but not surjective.
\textbf{ii)} $\Aqn$ has endomorphisms which are neither injective
nor surjective. \textbf{iii)} Any endomorphism of $\Atn$ and of
$\Atnf$ is injective.
\end{proposition}
\begin{proof} The representation $\C_{\mathbf{a}}$ defined above gives an algebra map
$\phi_{\mathbf{a}}: \Aqn \to \C = \End_\C(\C_{\mathbf{a}})$. If we
let $\epsilon: \C \to \Aqn$ be the inclusion we see that $\epsilon
\circ \phi_{\mathbf{a}}$ is a non-injective and non-surjective
$\C$-algebra endomorphism of $\Aqn$. This proves \textbf{ii)}.

Next, let $\phi \in \End(\At)$ (or $\End(\Aq)$) be defined by
$$
x \mapsto x, \ \pa \mapsto \pa + \f
$$
Clearly, $\phi$ is injective. A computation using \ref{frel} and
the PBW-basis shows that $\pa \notin \Jm \phi$, so $\phi$ is not
surjective. This proves \textbf{i)}.

Finally, let $P \in \Atn$ and assume that $\phi(P) = 0$. Since
$A_n = A^1_n$ is simple we see that specialization to $t=1$ gives
an injective map $\pi_1(\phi): A_n \to A_n$; thus $P = (1-t)Q$ for
some $Q \in \Atn$. Since multiplication by $(1-t)$ is injective on
$\Atn$ we get that $\phi(Q) = 0$. Repeating this procedure we
deduce that $P$ is divisible by $(1-t)^N$ for any $N
> 0$; hence $P = 0$. Thus endomorphisms of $\Atn$ are injective.
The same argument shows that endomorphisms of $\Atnf$ are
injective. This proves \textbf{iii)}.
\end{proof}

In general there is not unique factorization in the ring $\Atn$
(see \cite{CP}). However, the $\f_i$'s (which is precisely what we
will need) behave well with respect to factorization:
\begin{lemma}\label{f_i divides} [\cite{CP}, theorem 7] $\f_i$ is prime in $\Atn$ for each $i = 1, \ldots ,
n$, that is, $\f_i |\, ab \implies \f_i |\, a$ or $\f_i |\, b$,
for $a,b \in \Atn$.
\end{lemma}
Here $c | \, d$ stands for ``$c$ divides $d$ either from the left
or from the right". It follows from \ref{frelgeneral} that $\f_i$
divides $ab$ from the left $\iff$ $\f_i$ divides $ab$ from the
right.
\begin{proof}[Proof of lemma \ref{f_i divides}]
We can assume $i = 1$. Let $I$ be the two-sided ideal in $\Atn$
generated by $\f_1$. Since, by \ref{frelgeneral}, $I = \Atn \f_1 =
\f_1 \Atn$ it suffices to prove that $\Atn/I$ is an integral
domain; using that $\Atn = \At \ot_R \cdots \ot_R \At$ and that
$\At$ is an integral domain we reduce to the case $n=1$.

Let $R_{1-t}$ be the localization of the ring $R$ obtained by
inverting $1-t$. Then $R_{1-t}$ is flat over $R$ so $\At$ is a sub
ring of $A^{(t)}_{1-t} := \At \ot_R R_{1-t}$ so it is enough to
prove that $A^{(t)}_{1-t}/(I)$ is an integral domain. Using the
PBW basis of $A^{(t)}_{1-t}$ we see that an $R_{1-t}$-basis of
$A^{(t)}_{ 1-t}/(I)$ is given by
$$
\{x^m, \, \pa^l; \, m, l \in \N\}
$$
Moreover, we have $x\pa = \pa x = (1-t)^{-1}$ in $A^{(t)}_{
1-t}/(I)$ so we conclude that $A^{(t)}_{ 1-t}/(I)$ is isomorphic
to a ring of Laurent polynomials $R_{1-t}[X,X^{-1}]$ which is an
integral domain.
\end{proof}

\subsection{Azumaya property at a root of unity.}\label{Azumaya property at a root of unity.}
Let $\U \subset \C^*$ be the group of roots of unity; for $l \in
\N$ let $\Ul = \{q \in \C^*; q^l = 1\}$ and let $\Ulp \subset
\U_l$ be the primitive $l$'th roots of unity.

Let us fix a $q \in \Ulp$, $l >1$. A straightforward computation
using \ref{xcxw} and \ref{ewertg} shows that the center $\Zqn$ of
$\Aqn$ equals the polynomial ring
$$
\Zqn = \C[x^l_1, \ldots , x^l_n, \pa^l_1, \ldots , \pa^l_n].
$$
Moreover, by \ref{ewertg}, $\Aqn$ is free of rank $l^{2n}$ over
its center with basis $\{\x^\alpha \bpa^\beta; \alpha, \beta \in
\N^n, \, 0 \leq \alpha_i, \beta_i < n\}$. (If $q' \in \C^*$ is not
a root of unity or $q'=1$, then $A^{q'}_n$ has trivial center.)

Let $\AL(\Aqn)$ be the Azumaya locus of $\Aqn$. The following
result was found in collaboration with Natalia Pinz\'on Cort\'es.
\begin{proposition}\label{AL} Let $\m = (x^l_1-a_1, \ldots , x^l_n-a_n, \pa^l_1-b_1, \ldots , \pa^l_n-b_n) \in \mSpec \Zq$, $a_i, b_i \in \C$.
Then $\m \in \AL(\Aqn) \iff a_ib_i \neq (1-q)^{-l}$, for $1 \leq i
\leq n$.
\end{proposition}
\begin{proof} Since $\Aqn = A^q_1 \otimes_R A^q_1 \cdots \otimes_R A^q_1$
and the fiber (over a point in the spectra of its center) of a
tensor product of algebras is a full matrix algebras if and only
if the fiber of each tensor is a full matrix algebra (\cite{DI}),
it is enough to consider the case $n = 1$. Thus, we have $\Aq =
A^q_1$ and we write $a= a_1$ and $b = b_1$.

Assume that $\pi: \Aq/(\m) \rightarrow M_l(\C)$ is an isomorphism.
Put $X = \pi(x)$ and $Y = \pi(\pa)$. Assume first $a \neq 0$. Then
by assumption $X^l = a$ and $X$ is thus diagonalizable. We may
assume that $X = diag(\lambda_1, \ldots, \lambda_l)$ where
$\lambda^l_i = a$. We must have
\begin{equation}\label{ld}
M_l(\C) = \operatorname{Span} \{X^iY^j; 0 \leq i,j < l\}
\end{equation}
Counting dimensions it follows that $\{X^iY^j; 0 \leq i,j < l\}$
is a basis for $M_l(\C)$. Hence, in particular, $I,X,X^2, \ldots ,
X^{l-1}$ are linearly independent. Thus, all $\lambda_j$'s are
different. Fix $\lambda \in \C$ such that $\lambda^l = a$. After a
permutation base change we can assume that $\lambda_i = \lambda
\cdot \exp(2\pi \ii \cdot i/l)$.

We can assume $q= \exp(2\pi \ii /l)$. The equation $YX-qXY = 1$
shows that
$$Y=
       \left(
        \begin{array}{ccccccc}
          \frac{1}{\lambda_0(1-q)} & b_{01} & 0 & \ldots & 0 & 0 \\
          0 & \frac{1}{\lambda_1(1-q)} & b_{12} & \ldots & 0 & 0 \\
          \vdots & \vdots & \vdots & \vdots & \vdots & \vdots \\
          0 & 0 & 0  & \ldots & \frac{1}{\lambda_{l-2}(1-q)} & b_{l-2,l-1} \\
          b_{l-1,\,0} & 0 & 0 & \ldots & 0 & \frac{1}{\lambda_{l-1}(1-q)} \\
        \end{array}
       \right)
      $$
for some $b_{ij} \in \C$. One checks that $\C^l$ is an irreducible
representation of $\pi(\Aq/(\m))$ if and only if each $b_{ij} \neq
0$. Thus, by Burnside's theorem, \ref{ld} holds iff each $b_{ij}
\neq 0$.

On the other hand, using that $Y^l = b$ one verifies that each
$b_{ij} \neq 0$ holds if and only if $ab \neq (1-q)^{-n}$. This
proves the proposition when $a \neq 0$.

The case $b \neq 0$ is similar. The remaining case is $a = b= 0$.
In that case we note that truncated polynomials $\C[x]/x^{l}$ is
an irreducible representation of $\Aq/\m$ and the proposition
follows from  Burnside's theorem.
\end{proof}

\subsection{Azumaya-localization of $\Aqn$.}\label{section Azumaya loc} We keep a fixed $q \in
\Ulp$, $l > 1$.  Recall from \ref{frelgeneral} that $\{\f^k; k
\geq 0\}$ is a left Ore set, so that we can form the (left)
Ore-localization $\Atnf$ whose elements are represented by
expressions $P\f^{-k}$, $P \in \Atn$ and $k \geq 0$.

It follows (e.g., from lemma \ref{f_i divides}) that the units of
$\Atnf$ are precisely the elements of the form
\begin{equation}\label{units in Atnf}
c \cdot \prod^n_{i=1} \f_i^{m_i}
\end{equation}
for $m_i \in \Z$ and $c$ a unit in $R$. (The units in $R$ are of
the form $\lambda \cdot t^k$, for $\lambda \in \C^*$ and $k \in
\Z$.)

\smallskip

\noindent For any $q' \in \C^*$ we similarly have the
Ore-localization $\Aqnfprime$ of $\Aqnprime$. It follows from
\ref{frel} that $\flq \in \Zqn$. Therefor there is the
localization $\Zqnf$ and we observe that
\begin{equation}\label{center at f}
\Aqnf \cong \Aqn \ot_{\Zqn} \Zqnf
\end{equation}
\begin{lemma}\label{fm} We have $f^l_q = \prod^n_{i=1}1 -
(1-q)^lx^l_i\pa^l_i$.
\end{lemma}
\begin{proof} It is clearly enough to consider the case $n =1$. Since $\f^l_q \in
\Zq$ we conclude that $\f^l_q = a + bx^l +c\pa^l + dx^l \pa^l$. It
is easy to see that $a = 1$ and $d = -(1-q)^l$. Now $f^l = \sum_{0
\leq i,j \leq l} b_{ij} x^i \pa^j$ and an induction on $l$ shows
that $b_{0,l} = b_{l,0} = 0$. Thus $b = c = 0$ which proves
\ref{fm}.
\end{proof}

In view of this result we note that proposition \ref{AL} states
that $\AL(\Aqn) = U_{\flq}$, where
$$U_{\flq} \overset{def}{=} \{\m \in \mSpec(\Zqn); \flq \notin \m\} \cong \mSpec(\Zqnf)$$
is a standard open subset of $\mSpec(\Zqn)$. We have
\begin{lemma}\label{seclem} The center of $\Aqnf$ is
isomorphic to $\Zqnf$ and $\Aqnf$ is an Azumaya algebra.
\end{lemma}
\begin{proof} That the center of $\Aqnf$ is isomorphic to $\Zqnf$
follows from \ref{center at f}.

Since $\Aqnf$ is free over $\Zqnf$ and for each $\m \in
\max(\Zqnf) \subset \max(\Zqn)$ we have $\Aqnf/(\m) \cong
\Aqn/(\m)$ it follows from proposition \ref{AL} that $\Aqnf$ is
Azumaya.
\end{proof}

The endomorphism ring of $\Atnf$ is easy to describe:
\begin{proposition}\label{endomorphisms of A_t} Let $\phi \in \End_{R-alg}(\Atf)$.
Then there are units $a_i \in \Atnf$ and a $\sigma \in Sym_n$,
such that $\phi(x_i) = x_{\sigma_i} a_i $ and $\phi(\pa_i) =
a^{-1}_i \pa_{\sigma_i}$, for $i = 1, \ldots , n$. In particular,
$\phi$ is an automorphism and $\phi(\f) = \f$.
\end{proposition}
\begin{proof} \textbf{a)} Let $\phi \in \End_{R-alg}(\Atf)$.
Let $P_i = \phi(x_i)$, $Q_i = \phi(\pa_i)$, for $i = 1, \ldots ,
n$. Thus $[Q_i, P_i]_t = 1$. Since $\f_i$ is a unit in $\Atf$ we
must have that $\phi(\f_i)$ is a unit in $\Atf$. Thus, by
\ref{units in Atnf} we have
\begin{equation}\label{edf} 1-(1-t)P_iQ_i = \phi(\f_i) =  c_i \cdot \prod^n_{j=1} \f_j^{m_{ij}}
\end{equation}
where $c_i \in R$ is a unit and $m_{ij} \in \Z$. Since the right
hand side of this equation commutes with $\f_j$, we see that
$P_iQ_i$, and therefor also $Q_iP_i$, commute with $\f_j$, for
each $i,j$.

\smallskip

\noindent \textbf{b)} Take integers $b_{ij} \in \Z$ such that $$\w
P_i := \prod^n_{j=1} \f_j^{-b_{ij}}P_i \in \Atn$$ and $\w P_i$ is
not divisible by $\f_j$ in $\Atn$ for any $j$ (recall that being
divisible by $\f_j$ from left and from right are equivalent
properties in $\Atn$).

Let $\w Q_i = Q_i\prod^n_{j=1} \f_j^{b_{ij}}$. By \textbf{a)} we
get $[\w Q_i, \w P_i]_t = 1$ and from this and lemma \ref{f_i
divides} it follows that $\w Q_i \in \Atn$. Again by \textbf{a)}
we have
\begin{equation}\label{edf2} 1-(1-t)\w P_i \w Q_i = c_i \cdot
\prod^n_{j=1} \f_j^{m_{ij}}
\end{equation}
We conclude that each $m_{ij} \geq 0$.
\smallskip

\noindent \textbf{c)} We now prove that for each $i$ there is at
most one integer $j$ such that $m_{ij} > 0$. We fix $i$ and assume
to get a contradiction that there are two different integers $j$
and $j'$  such that $m_{ij}, m_{ij'} > 0$.

Let $\C[[1-t]]$ be the ring of formal power series in the variable
$1-t$; thus we have an embedding $R \hookrightarrow \C[[1-t]]$ and
we can define $\widehat{\Atn} := \Atn \ot_R \C[[1-t]]$. Thus
$\Atn$ is a subring of $\widehat{\Atn}$ and an element in $h \in
\widehat{\Atn}$ can be uniquely written as $h =
\sum^{\infty}_{k=0} h^{(k)} (1-t)^k$, where $h^{(k)} \in
\operatorname{span}_{\C}\{\x^\alpha\bpa^\beta; \alpha, \beta \in
\N^n\}$ is the coefficient of $h$ in front of $(1-t)^k$.

Now, think of \ref{edf2} as an equality in $\widehat{\Atn}$. Then
taking coefficients in front of (1-t) gives
$$
{\w P}^{(0)}_i {\w Q}^{(0)}_i = c^{(1)}_i - \sum^n_{j=1} m_{ij}
x_j \pa_j.
$$
We claim that $c^{(1)}_i - \sum^n_{j=1} m_{ij} x_j \pa_j$ is
irreducible in $\At$, i.e. if is written as a product of two
elements than one of them is a scalar in $R$. To see this it is
enough to see that $\pi_1(c^{(1)}_i - \sum^n_{j=1} m_{ij} x_j
\pa_j)$ is irreducible in $A^1_n$ and for this it is in turn
sufficient to see that
$$\sum^n_{j=1} m_{ij} X_j Y_j = \gr( \pi_1(c^{(1)}_i - \sum^n_{j=1} m_{ij} x_j
\pa_j))$$ is irreducible in $\C[X_1, \ldots, X_n, Y_1, \ldots ,
Y_n] = \gr(A^1_n)$. Here the associated graded is taken with
respect to Bernstein's filtration on $A^1_n$ (the filtration such
that $\deg x_i = \deg \pa_i = 1$). But it is well-known that
$\sum^n_{j=1} m_{ij} X_j Y_j$ is an irreducible polynomial, since
two different $m_{ij}$'s are non-zero. Thus we get without loss of
generality that ${\w P}^{(0)}_i = c^{(1)}_i - \sum^n_{j=1} m_{ij}
x_j \pa_j$ and ${\w Q}^{(0)}_i = 1$. But this implies that $1 =
[{\w Q}_i,{\w P}_i]_t \in (1-t) \Atn$ which is a contradiction.

\smallskip

\noindent \textbf{d)} By proposition \ref{naively false D}
\textbf{iii)} $\phi$ is injective and from this it follows that we
cannot have $m_{ij} = 0$ for all $j = 1, \ldots , n$. Thus there
is a unique integer $\sigma_i$ such that $m_{i\sigma_i} > 0$.
Since $\phi$ is injective we conclude that the $\sigma_i$'s define
a permutation $\sigma \in Sym_n$.

\smallskip

\noindent \textbf{e)} We now prove $m_{i \sigma_i} = 1$. If $m_{i
\sigma_i} > 1$ pick any $q' \in \U^{prim}_{m_{i \sigma_i} }$.
Since $\f^{m_{i \sigma_i}}_{q'} \in Z^{q'}_n$ it follows that $\w
Q_{i,q'}\w P_{i,q'} \in Z^{q'}_n$. Multiplying the equation $[{\w
Q}_{i,q'}, {\w P}_{i,q'}]_{q'} = 1$ with $\w P_{i,q'}$ we get
$$\ P_{i,q'} = [{\w
Q}_{i,q'}, {\w P}_{i,q'}]_{q'}\w P_{i,q'}  = [\hbox{ since } \w
Q_{i,q'}\w P_{i,q'}\hbox{ is central }] = \w P_{i,q'}\w
Q_{i,q'}(1-q')\w P_{i,q'}$$ so that $\w P_{i,q'}\w Q_{i,q'} =
(1-{q'})^{-1}$. But then $\phi(\f_i)_{q'} =
1-(1-{q'})(1-{q'})^{-1} = 0$ which contradicts \ref{edf2}.

\smallskip

\noindent \textbf{f)} We now know that $\w P_i, \w Q_i \in \At$
and that $1-(1-t)\w P_i \w Q_i = c_i \cdot \f_{\sigma_i}$. Thus
the product $\w P_i \w Q_i$ has degree $2$ with respect to the
Bernstein filtration on $\Atn$. It is clear that neither $\w P_i$
nor $\w Q_i$ can be a scalar (e.g. since $\phi$ is injective) so
we conclude that $\w P$ and $\w Q$ have degree $1$.

From this and the fact that $[\w P, \w Q]_t = 1$ a trivial
computation shows that $\w P_i = v_i x_{\sigma_i}$ and $\w Q_i =
v^{-1}_i \pa_{\sigma_i}$ for some unit $v_i \in R$. This implies
that $P_i = x_{\sigma_i} a_i$, $Q_i = a^{-1}_i \pa_{\sigma_i}$,
where $a_i = v_i \cdot \prod^n_{j=1} f^{b_{ij}}$, and we conclude
that $\phi(\f_i) = \f_{\sigma_i}$.

Since $\sigma$ is a permutation we get that $\phi$ is an
automorphism and that $\phi(\f) = \f$.

\end{proof}

Notice that proposition \ref{endomorphisms of A_t} implies that
automorphisms of $A_n$ in general cannot be lifted to
automorphisms of $\Atnf$. For example, the automorphism $\pa
\mapsto \pa + 1$ and $ x \mapsto x$ of $A_1$ do not lift to an
automorphism of $\Atnf$. (See section \ref{known structure groups}
for a description of $\Aut(A_1)$.)

It is possible that any automorphism of $\Atn$ is of the form
$$
x_i \mapsto a_i x_{\sigma_i} , \pa_i \mapsto a^{-1}_i
\pa_{\sigma_i}
$$
for some $a_i \in \C^*$ and $\sigma \in Sym_n$. But we haven't
been able to prove this.

\subsection{Endomorphisms of $\Aqn$ preserve the center}\label{Endomorphisms of Aqn preserve the center}
In the introduction we quoted theorem A) which asserts that an
endomorphism of an Azumaya algebra automatically preserves the
center.

We do not know to what extend this result may generalize to almost
Azumaya algebras, but for the quantized Weyl algebra $\Atn$ we do
have a positive result. This is obtained by modifying the
arguments of \cite{ACvE} and \cite{B-KK} and using the Azumaya
localization $\Atnf$.

\begin{theorem}\label{almost preserve center} Let $\phi \in \End_{R-alg}(\Atn)$. For almost every root of unity $q$ we have $\phi_q(\Zqn) \subseteq \Zqn$.
\end{theorem}
\emph{Proof.} \textbf{a)} Let $q$ be a root of unity and consider
the composition
$$\psi_q: \Zqn \hookrightarrow \Aqn
\overset{\phi_q}{\to}  \Aqn \overset{\pi_q}{\to} \Aqnf$$ where
$\pi_q$ is localization. Let
$$D_q = \{\m \in \max \Zqnf; \Zqn \cdot \f^l_q + \phi^{-1}(\Aqnf \m) =
\Zqn\}.$$ \emph{Claim:} For almost all $q$, $D_q$ is a
Zariski-open and dense subset of $\max \Zqnf$.

It is clear that $D_q$ is Zariski-open. Hence it suffices to show
that $D_q \neq \varnothing$. Assume that $\m \notin D_q$, then we
have that $\Zqnf \cdot \f^l_q + \psi^{-1}_q(\Aqnf \m)$ is a proper
ideal. By the nullstellensatz this implies $\psi^{-1}_q(\Aqnf \m)
\subseteq V(\f)$, i.e. that there is an integer $k$, such that
$\psi_q((\f^l_q)^k) \in \Aqnf \cdot \m$. Apriori $k$ depends on
$\m$, but since $\Aqnf/\Aqnf \cdot \m \cong M_n(\C)$ and the
maximal nilpotency degree of any element in $M_n(\C)$ is $n$ we
see that in fact $\psi_q((\f^l_q)^n) \in \Aqnf \cdot \m$. Thus
$$
\psi_q((\f^l_q))^n \in \cap_{\m \notin D_q} (\Aqnf \cdot \m) =
\Aqnf(\cap_{\m \notin D_q} \m)
$$
where the last equality follows from well-known properties of
Azumaya algebras applied to $\Aqnf$. Hence, since the intersection
of all maximal ideals in a polynomial ring is $0$, we either have
that $D_q$ is non-empty or $\psi_q(\f^l_q)^n = 0$. In the latter
case we have $\phi_q(\f_q) = 0$. Now, if $\phi_q(\f_q) = 0$ for
infinitely many roots of unity $q$ then clearly $\phi(\f) = 0$.
This would mean that
$$\prod^n_{i=1} 1-(1-t)\phi(x_i) \phi(\pa_i)
= 0$$ which contradicts the fact that $1$ is not divisible by
$(1-t)$ in $\Atn$. This proves the claim.

\smallskip

\noindent \textbf{b)} Let $q$ be a root of unity such that $D_q$
is dense in $\max \Zqnf$. Let $a \in \Aqn$ and suppose that
$\phi_q(a)$ is not central in $\Aqn$. Then the image of
$\pi_q\circ \phi_q(a)$ is not central in $\Aqnf$. Take $b \in
\Aqnf$ such that $c := [b, \pi_q \circ \phi_q(a)] \neq 0$. Since
$\Aqnf$ is free of finite rank over $\Zqnf$ we see that the set of
$\m \in \max \Zqnf$ such that $\overline{c} \neq 0$ in
$\Aqnf/\Aqnf \m$ is Zariski open and non-void. Thus we can find
such an $\m$ which in addition is in $D_q$. Since $\m \in D_q$ we
have that $\Aqn/\phi^{-1}_q\circ \pi^{-1}_q(\Aqnf \m)$ is Azumaya
over $\Zqn/\psi^{-1}_q(\Aqnf \m)$. Hence there is a maximal ideal
$\m' \in \Zqnf$ such that $\phi_q$ induces an isomorphism
$\overline{\phi}_q: \Aqn/\Aqn \m' \overset{\backsim}{\to}
\Aqnf/\Aqnf \m$. Since both sides are isomorphic to $M_n(\C)$ it
follows from standard properties of matrix rings that $\bar{a}$ is
not central in $\Aqn/\Aqn \m'$. Thus $a \notin \Zqn$. $\Box$.
\begin{remark}
The usual Weyl algebra is a quantization-deformation of a
commutative polynomial ring in $2n$ variables so a quantized Weyl
algebra is actually a second quantization. However, it is known
that the second Hochschild cohomology group $HH^2(A_n,A_n) = 0$,
(see \cite{B-KK}), so $\Atn$ is a trivial formal deformation of
$A_n$ (i.e., if we replace our base ring $R$ by its extension
$\C[[1-t]]$). Because of this it is not evident how one should
conceptually characterize quantized Weyl algebras.

They can be interpreted as Grothendieck rings of differential
operators on quantum spaces using the formalism of braided tensor
categories, see \cite{Majid} and references therein, or, in direct
relation with the representation theory of quantum groups, as
differential operators on Bruhat cells of quantized flag
manifolds, see \cite{J} and \cite{JL}. In the latter context
another application of the Azumaya property of quantized Weyl
algebras to Beilinson-Bernstein localization was given in
\cite{BK}. Note that the quantized Weyl algebras that occur in
relation to representation theory are in general more complicated
than the ones considered in this paper. In fact, we consider in
this note precisely (tensor powers of) the ones that relate to the
representation theory of $U_q(\mathfrak{sl}_2)$.
\end{remark}
\section{Lifting endomorphisms of $A_n$ to $\Atn$}
Here we discuss the problem of lifting an endomorphism $\phi$ of
$A_n$ to an endomorphism $\widetilde{\phi}$ of $\Atn$ (i.e., $\w
\phi$ satisfy $({\w \phi})_1 = \phi$); we conjecture that this is
always possible. The structure of the group $\Aut(A_1)$ is known
(and so is $\Aut(P_1)$). We use this to show that it is possible
to lift any automorphism of $A_1$.
\subsection{The structure of the groups $\Aut(A_1)$ and $\Aut(P_1)$}\label{known structure groups}
Let us start by recalling the structure of the relevant
automorphism groups in the only case where they are known, namely
for $n= 1$.  Let $\F$ be any field and let $H \cong \F[x]$ be the
subgroup of $\Aut(A_1(\F))$ consisting of translations
$$x \mapsto x, \, \pa \mapsto \pa + G(x).$$
for $G(x) \in \F[x]$ and let $K \cong SL_2(\F) \ltimes \F^2$ be
the subgroup of $\Aut(A_1(\F))$ consisting of elements $$ \left(
\begin{array}{cc}
a& b  \\
c & d \\
 \end{array} \right) \times \left(
\begin{array}{cc}
r  \\
s \\
 \end{array} \right): x \mapsto ax+b\pa + r, \; \pa \mapsto cx+d\pa + s.$$
 Then $\Aut(A_1(\F))$ is isomorphic to the amalgamated product of $H$
 and $K$ over their intersection. (For a proof see \cite{D} and \cite{M-L}.)
 There is a similar description of $\Aut(P_1(\F))$, given in \cite{Jung}; thus $\Aut(A_1(\F)) \cong \Aut(P_1(\F))$.

 It follows from the above that $\Aut(A_1(\F))$ is generated by
 \begin{equation}\label{phi, psi}
\phi_{F(x)}: x \mapsto x, \pa \mapsto \pa + F(x), \psi_{G(\pa)}: x
\mapsto x + G(\pa) \hbox{ and } \pa \mapsto \pa
\end{equation}
for  $F(x) \in \F[x], G(\pa) \in \F[\pa]$. Moreover, since
$\phi_{F_1(x)} \circ \phi_{F_2}(x) = \phi_{F_1(x) + F_2(x)}$, it
is enough to let the $F(x)$'s and $G(\pa)$'s run over set of
generators of the additive groups $\F[x]$ and $G[\pa]$,
respectively, to get a generating set for $\Aut(A_1(\F))$ .
Similarly with $\Aut(P_1(\F))$.

\subsection{Lifting conjecture}\label{Lifting conjecture}
\begin{conjecture}\label{ConjA} Any $\phi \in \End(A_n)$
has a lift $\widetilde{\phi} \in \End(\Atn)$.
\end{conjecture}

The evidence we have for this conjecture is that every
automorphism of $A_1$ lifts to an endomorphism of $\At$. It is
enough to construct lifts of the elements $\phi_{F(x)}$ and
$\psi_{G(\pa)}$ from \ref{phi, psi}. This is done by
$\widetilde{\phi}_{F(x)}: x \mapsto x, \pa \mapsto \pa + F(x)\cdot
\f$, $\widetilde{\psi}_{G(\pa)}: x \mapsto x + \f \cdot G(\pa),
\pa \mapsto \pa$.

The lift, if it exists, is not unique of course, e.g., $\Id_{\At}$
and $x \mapsto x, \pa \mapsto \pa + (t-1)\f$ are two lifts of
$\Id_{A_1}$.

Note that if $\operatorname{D}_1$ is true it follows that any
endomorphism of $A_1$ can be lifted to an endomorphism of
$A^{(t)}_1$. This is one good motivation to try to directly prove
the, we suspect, difficult assertion that any endomorphism of
$A_1$ can be lifted. One may speculate that the obstruction groups
to lifting problems that exist in homotopy theory, e.g.,
\cite{CDI}, applied, say, to the model category of dg-algebras,
could help resolving this conjecture; ideally the vanishing of the
relevant obstruction groups could then be derived from
$\mathcal{D}$-module theory.

\section{Quantum approach to Belov-Kanel and Kontsevich's
conjecture} In this section we construct a monoid homomorphism
$\;\widehat{\ }\;: \End(\Atn)' \to \End(P_n)$ where $\End(\Atn)'$
is a certain submonoid of $\End(\Atn)$ and give a conjectural
description of the image of this map. Moreover, given any $\phi
\in \Aut(A_n)$ we sharpen conjecture \ref{ConjA} to conjecture
that there is a lift $\widetilde{\phi}$ in $\End(\Atn)'$.

\subsection{Poisson bracket on $\Zqn$} Just like a deformation
quantization of an algebra (see \cite{K}) or a Weyl algebra in
finite characteristic (see \cite{B-KK2}) admit canonical Poisson
brackets on their centers, we shall construct a Poisson bracket
\begin{equation}\label{Poisson}
\{\; , \; \}_q: \Zqn \times \Zqn \to \Zqn
\end{equation}
when $q$ is a root of unity. This bracket will be degenerate for
any $q$ but in the limit as $q \to 1$ it will be isomorphic to the
standard symplectic bracket on a polynomial ring in $2n$
variables.

To construct $\{\; , \; \}_q$ we proceed as follows. Let $q \in
\Ulp$, $l > 1$. For $P \in \Aqn$, let $\widetilde{P} \in \Atn$
denote any lift of $P$, i.e. $\pi_q(\w P) = P$.

For $P,Q \in \Zqn$ we now define
\begin{equation}\label{Poissondef}
\{P,Q\}_q = \lambda_q \cdot \pi_q({1\over{t-q}}[\widetilde{P},
\widetilde{Q}])
\end{equation}
where $\lambda_q = \pi_q({
 t-q \over {[l]_t!}})$. Note that
$[\widetilde{P}, \widetilde{Q}]$ is divisible by $t-q$ because,
since $P$ is central, we have
$$\pi_{q}([\widetilde{P}, \widetilde{Q}]) =
[\pi_{q}(\widetilde{P}),\pi_{q}(\widetilde{Q})] = 0.$$ A similar
computation shows that $\{P,Q\}_q$ is independent of chosen lifts.
Thus $\{P,Q\}_q$ is well-defined. The Jacobi-identity implies that
$\{P,Q\}_q$ is a central element.

A straight forward computation shows that $\{\; , \;\}_q$
satisfies the axioms of a Poisson bracket, i.e. that it is a Lie
bracket which is a $\C$-linear derivation in each factor with the
other factor fixed.
\begin{lemma}\label{straight}
For $q \in \Ulp$, $l > 1$, we have
$$\{\pa^l_i, x^l_i\}_q = 1 + {l(q-1) \over [l-1]_q!}
x^l_i\pa^l_i$$ for $i = 1, \ldots n$.
\end{lemma}
\begin{proof} We can assume $n=1$ and write $\pa = \pa_1$, $x = x_1$.
We get an equality in $\At$ of the form
$$[{ \pa}^l,{ x}^l] = \sum^l_{i=0} a_{i,l} { x}^i { \pa}^i, \; a_{i,l} \in R.$$
From this it follows that in $\Aq$ holds
$$
\{\pa^l,x^l \}_q = \alpha \cdot 1 + \beta \cdot x^l \pa^l
$$
for some $\alpha, \beta \in \C$ and we must show that $\alpha = 1$
and $\beta = {l(q-1) \over [l-1]_q!}$.

It is easy to see that $a_{l,l} = t^{l^2}-1$. Since, $t^{l^2}-1 =
(t^l-1)(1 + t^l + \ldots + t^{l^2-l})$ and $q^l = 1$ we get
$$\beta = \lambda_q \cdot \pi_q({t^{l^2}-1 \over t-q}) = \lambda_q \cdot l(q-1)\cdot \pi_q( {[l]_t \over t-q}) =
{l(q-1) \over [l-1]_q!}.$$

Let $R_{1-t}$ be the localization of $R$ obtained by inverting
$1-t$ and put $\At_{1-t} = \At \otimes_R R_{1-t}$; thus $\At$ is a
subring of $\At_{1-t}$. For any $P \in \At_{1-t}$ we denote by
$c(P) \in R_{1-t}$ its constant term with respect to the basis
$\{x^i\pa^j; i,j \geq 0\}$ of $\At_{1-t}$ over $R_{1-t}$. In
$\At_{1-t}$ we have the equalities ${ \pa}^l{ x}^l =$
$$
{ \pa}^{l-1}(\pa x) x^{l-1} = { \pa}^{l-1}({t\f \over t-1} -{1
\over t-1}){ x}^{l-1} = {t^l\f -1 \over t-1}{ \pa}^{l-1}{
x}^{l-1}.$$ Thus, since $\f$ commutes with ${ \pa}^{l-1}{
x}^{l-1}$ and $c(\f) = 1$, we have
$$a_{0,l} \overset{def}{=}  c({ \pa}^{l} { x}^{l})
= [l]_t \cdot c({ \pa}^{l-1} { x}^{l-1}) = [l]_t \cdot a_{0,l-1}
$$
Since $a_{0,1} = 1$ this shows $a_{0,l} = [l]_t!$. Thus, $ \alpha
= \lambda_q \cdot \pi_q({[l]_t! \over t-q}) = 1.$
\end{proof}
\smallskip

Let $q_l = \exp(2\pi \ii \cdot {1\over l})$, $l \geq 0$. Consider
the standard symplectic bracket $\{\; , \; \}_{\BGG_{2n}}$ on the
ring $\BGG_{2n} = \C[r_1, \ldots , r_n, s_1, \ldots s_n]$ defined
by $\{\phi, \psi\}_{\BGG_{2n}} = \sum^n_{i=1} {\pa\phi \over \pa
r_i} {\pa\psi \over \pa s_i} - {\pa\psi \over \pa r_i} {\pa\phi
\over \pa s_i}$. Let
$$\Theta_l : \BGG_{2n} \overset{\sim}{\to}
Z^{q_l}_n$$
be the algebra isomorphism defined by $\Theta_l(r_i) =
\pa^l_i$ and $\Theta_l(s_i) = x^l_i$. We say that a sequence
$\psi^{(j)} = \sum a^{(j)}_{\alpha, \beta} \x^\alpha \bpa^\beta
\in \Atm$, $j \geq 1$, converges to $\psi = \sum a_{\alpha, \beta}
\x^\alpha \bpa^\beta \in \Atm$ if each coefficient
$a^{(j)}_{\alpha, \beta}$ converges to $a_{\alpha, \beta}$ in the
analytic topology of $\C$. We then have
\begin{proposition}\label{transport} Define a bracket $\{\;, \;\}$ on $\BGG_{2n}$, by
$$
\{P,Q\} = \lim_{l \to \infty} \Theta^{-1}_l(\{\Theta_l(P),
\Theta_l(Q)\}_{q_l}).
$$
Then the bracket $\{\;, \;\}$ is equal to the standard Poisson
bracket $\{\;, \;\}_{\BGG_{2n}}$.
\end{proposition}
\begin{proof} In order to establish convergence of $\{\;,\;\}$ and the
prescribed equality of brackets it is enough to prove that
$\{r_i,s_i\} = 1$, for $1 \leq i \leq n$ (because it is trivially
true that the bracket $\{ \ , \ \}$ vanish on all other pairs of
elements in the set $\{r_1, \ldots, r_n, s_1, \ldots , s_n\}$.) We
have $\{r_i, s_i\} =$
$$
\lim_{l \to \infty} \Theta_l^{-1}(\{\Theta_l(r_i),
\Theta_l(s_i)\}_{q_l}) = \lim_{l \to \infty}
\Theta_l^{-1}(\{\pa^l_i, x^l_i\}_{q_l}) = \hbox{[by
\ref{straight}]} =
$$
$$
\lim_{l \to \infty} \Theta^{-1}_l(1+  {l(q-1) \over
\cdot[l-1]_{q_l}!}x^l_i\pa^l_i) = \lim_{l \to \infty} 1+  {l(q-1)
\over [l-1]_{q_l}!}s_i r_i = 1
$$
The last inequality holds since $\lim_{l \to \infty} l(q_l-1) =
2\pi \ii$ and $\lim_{l \to \infty} [l-1]_{q_l}! = \infty$, because
for any $1 \leq a \leq l-1$, we have $\vert [a]_{q_l}\vert \geq 1$
and for ${l-1 \over 2} \leq a \leq l-{l-1 \over 2}$, we have
$\vert [a]_q \vert \geq 2$, if say $l > 6$.
\end{proof}

\subsection{The map $\;\widehat{\ }\;: \End(\Atn)' \to \End(P_n)$.}

Let $\psi \in \End(\Atn)$. By theorem \ref{almost preserve center}
we know that for almost all roots of unity $q$ we have
$\psi_q|_{\Zqn} \in \End(\Zqn)$.

We say that $\psi$ has degree $\leq N$ (with respect to the
Bernstein filtration) if $$\deg \psi(x_i), \deg \psi(\pa_i) \leq
N, \ \forall i.$$

If $\psi_q|_{\Zqn} \in \End(\Zqn)$, $q^l = 1$, we see that $
\Theta^{-1}_l\psi_{q_l}(\Theta_l(r_i))$ and
$\Theta^{-1}_l\psi_{q_l}(\Theta_l(s_i))$ are polynomials of total
degree $\leq N$ for all $i$. Hence, if $P \in \BGG_{2n}$ ha degree
$\leq m$ we see that $\Theta^{-1}_l\psi_{q_l}(\Theta_l(P))$ is a
polynomial of degree $\leq mN$.

We define $\widehat{\psi}  \in \End(\BGG_{2n})$, by
$$
\widehat{\psi}(P) = \lim_{l \to \infty, \; l \, prime}
\Theta^{-1}_l\psi_{q_l}(\Theta_l(P)), \ P \in \BGG_{2n}
$$
if each coefficient of this polynomial converges in the analytic
topology of $\C$.
\begin{proposition}\label{properties of hat} \textbf{a)} If $\hat{\psi}(P), \hat{\psi}(Q)$ converges for some
$P,Q \in \BGG_{2n}$ then $\hat{\psi}(\{P,Q\})$ also converges and
equals $\{\hat{\psi}(P), \hat{\psi}(Q)\}$. \textbf{ b)} If
$\hat{\psi}(x_i)$ and $\hat{\psi}(\pa_i)$ converges for all $i$,
then $\hat{\psi}(P)$ converges for all $P \in \BGG_{2n}$. \textbf{
c)} If $\hat{\psi}$ and $\hat{\phi}$ converges then $\widehat{\psi
\circ \phi} = \widehat{\psi} \circ \widehat \phi$ converge.
\textbf{ d)} If $\psi \in \Aut(\Atn)$ and $\widehat \psi$
converges then $\widehat \psi$ is invertible and $\widehat
\psi^{-1} = \widehat{\psi^{-1}}$ converges.
\end{proposition}
\begin{proof}  \textbf{a)} holds since the Poisson bracket on
$\Zqn$ was defined in terms of the Lie bracket on  the associative
ring $\Atn$ and because of proposition \ref{transport}.

\textbf{ b)} holds since each $\Theta_l$ is a homomorphism.
\textbf{ c)} and \textbf{ d)} are standard.
\end{proof}
Let $\Aut(\Atn)'$ (resp., $\End(\Atn)'$) be the subgroup of
$\Aut(\Atn)$ (resp., submonoid of $\End(\Atn)$) consisting of
convergent $\psi$'s. Thus we have a morphism of monoids
\begin{equation}\label{ring homomom}
\;\widehat{\ }\;: \End(\Atn)' \to \End(P_n)
\end{equation}
\subsection{Main conjectures}\label{Main conjectures}
We attempt to describe the image of $\;\widehat{\ }\;$ of
\ref{ring homomom}. The motivation for this conjectural
description will be based on an explicit computation for the case
$n = 1$ in the next section. (Let us remark that the group
$\Aut(\Atn)'$ seems to be too small to work with.)

Recall that conjecturally $\End(P_n) = \Aut(P_n)$ holds, so we
would expect that  $\;\widehat{ \ }\;$ takes values in the group
$\Aut(P_n)$. For any unitary commutative ring $\Lambda$ and $N \in
\N$, let $\Aut(P_n(\Lambda))^{\leq N}$ be the of all elements $g
\in \Aut(P_n(\Lambda))$ such that the polynomials $g(r_i), g(s_i),
g^{-1}(r_i)$ and $g^{-1}(s_i)$, $1 \leq i \leq n$,  all have total
degree $\leq N$. Let us remark that one can show that the
assignments $\Lambda \mapsto \Aut(P_n(\Lambda))^{\leq N}$ define
an ind-group scheme $\underline{\Aut}({P}_n)$ of (ind-)finite type
over $\Z$ such that $\underline{\Aut}({P}_n)(\Lambda) =
\Aut(P_n(\Lambda))$, see \cite{B-KK2}.
\begin{conjecture}\label{hat is surjective} The image of the map $\;\widehat{\ }\;$ of \ref{ring homomom} is contained in $\Aut(P_n(\Z))$.
Moreover, for any $N \in \N$, the composition
$$\Aut(P_n(\Z))^{\leq  N} \hookrightarrow  \Aut(P_n(\Z)) \to
\Aut(P_n(\Z))/\Jm \;\widehat{\ }\;$$ has finite image.
\end{conjecture}
This conjecture would imply that $\Jm \;\widehat{\ }\;$ is Zariski
dense in $\Aut(P_n)$ (in the ``ind-sense"). It is also plausible
that $\Jm \;\widehat{\ }\; = \Aut(P_n(\Z))$.

We give some variants of the lifting conjecture \ref{ConjA}.
\begin{conjecture}\label{ConjA'} Any $\phi \in \End(A_n)$
has a lift $\widetilde{\phi} \in \End_{R-alg}(\Atn)'$ .
\end{conjecture}

\begin{conjecture}\label{ConjA''} For any $g \in \Jm \;\widehat{\ }\;$ there is a $\phi \in \Aut(A_n(\Z))$
which has a lift $\widetilde{\phi} \in \End(\Atn(\Z))'$ such that
$g = \widehat{\widetilde{\phi}}$.
\end{conjecture}

The moral conclusion we would like to draw from an affirmative
answer to these conjectures is that there should be a homomorphism
$\Aut(A_n) \to \Aut(P_n)$ which is (almost) surjective (and if not
over $\C$-points at least over $\Z$-points).

\section{The case $n=1$}\label{evidence at 1 section}  In this
section we analyze what happens when we apply $\;\widehat{ \ }\;$
to lifts of elements of $\Aut(A_1)$.

\subsection{} It follows from \ref{phi, psi} and the discussion below
it that for any field $\F$ the group  $\Aut(A_1(\F))$ is generated
by elements of the form
\begin{equation}\label{first gen}
\phi_{\lambda x^m}: x \mapsto x, \ \pa \mapsto \pa + \lambda x^m
\hbox{ and }
\end{equation}
\begin{equation}\label{second gen}
\psi_{\lambda \pa^m}: x \mapsto x + \lambda \pa^m, \ \pa \mapsto
\pa,
\end{equation}
for $m \in \Z$ and $\lambda \in \F$ and we have similar generators
for $\Aut(P_1(\F))$. Let $G \subset  \Aut(A_1(\Z))$ be the
subgroup generated by elements $\phi_{x^m}$ and $\psi_{\pa^m}$ for
$m \in \Z$ and let $H \subset  \Aut(P_1(\Z))$ be the subgroup
generated by elements $\phi_{r^m}$ and $\psi_{s^m}$ for $m \in
\Z$. One can show that the images of the compositions
\begin{equation}\label{im 1}
\Aut(A_1(\Z))^{\leq N} \to \Aut(A_1(\Z)) \to \Aut(A_1(\Z))/G
\end{equation}
\begin{equation}\label{im 2}
\Aut(P_1(\Z))^{\leq N} \to \Aut(P_1(\Z)) \to \Aut(P_1(\Z))/H
\end{equation}
are finite sets for any $N \in \N$. This follows by considering
the corresponding sequences over $\F = \Q$ (in which case both
ending terms are $\{e\}$) and clearing denominators.


\begin{proposition} \label{evidence at 1} Let $\phi = \phi_{\lambda x^m} \in \Aut(A_1)$ and let
$\widetilde{\phi} \in \End(\At)$ be the lift defined by $x \mapsto
x, \, \pa \mapsto \pa + \lambda x^m\f$. Then
$\widehat{\widetilde{\phi}}$ converges iff
 $L :=
\lim_{l \to \infty,\, l \; prime} \lambda^l$ converges in $R$ and
in this case $\widehat{\widetilde{\phi}}(r) = r+ L s^m$ and
$\widehat{\widetilde{\phi}}(s) = s$. A similar result holds for
$\psi_{\lambda \pa^m}$.
\end{proposition}
\begin{proof}
For $a,b$ elements of an associative ring $A$ and $i, j \in \N$
define $W(a,i;b,j)$ to be the set of all words in $a$ and $b$ in
which $a$ has $i$ occurrences  and $b$ has $j$ occurrences. Thus
in $A$ we have
$$
(a+b)^n = \sum_{i+j = n,\, w \in W(a,i;b,j)} w
$$
Consider now the case $A = \At$, assume that $l > m+1$ is prime,
$0 \leq i \leq l$ and $w \in W(\lambda x^m\f,i; \pa,l-i)$. We
shall prove that
\begin{equation}\label{unless}
i \notin \{0,l\} \implies w_{q_l} = 0
\end{equation}
Since $l$ is prime there are two cases: \textbf{1)} $l > (m+1)i$
and \textbf{2)} $l < (m+1)i$.

In case \textbf{1)} we a priori have that  $w$ is a combination of
the elements $$x^{(m+1)i} \pa^l, x^{(m+1)i-1} \pa^{l-1}, \ldots,
\pa^{l-(m+1)i}.$$ Since $w_{q_l} \in Z^{q_l}$ we get $w_{q_l} =
\lambda x^{(m+1)i} \pa^l$, for some $\lambda \in R$, and that $l |
(m+1)$ so that $i = 0$ in this case.

In case \textbf{2)} we a priori have that  $w$ is a combination of
the elements $$x^{(m+1)i} \pa^l, x^{(m+1)i-1} \pa^{l-1}, \ldots,
x^{(m+1)i-l}.$$ Again, since $w_{q_l} \in Z^{q_l}$ we get $w_{q_l}
= \lambda x^{(m+1)i-l}$, for some $\lambda \in R$, and that $l |
(m+1)i-l$ so that $i = l$ in this case. This proves \ref{unless}.

It follows from \ref{unless} that
$$
(\pa + \lambda x^m\f)^l_{q_l} = \pa^l_{q_l} +
\lambda^l(x^m\f)^l_{q_l} = \pa^l_{q_l} +
\lambda^lx^{ml}_{q_l}\f^l_{q_l}
$$
where the last equality follows from $(x^m\f)^l_{q_l} =
{q_l}^{{(l-1)}l/2}x^{ml}\f^l_{q_l}$ and ${q_l}^{{(l-1)}l/2} =1$,
since $l$ is odd and $q_l$ is an $l$'th root of unity.

The proposition now follows from the fact that $\lim_{l \to
\infty} \f^l_{q_l} = 1$.
\end{proof}
This immediately implies
\begin{cor}\label{image at 1} The image of $\;\widehat{\ }$ contains the group $H$.
\end{cor}
The proposition also has the following seemingly unfortunate
consequence
\begin{cor}\label{different lifts same result} There is a $\phi \in \Aut(A_1)$ and two different
convergent lifts $\w \phi$ and $\w \phi'$ such that $\widehat{\w
\phi} \neq \widehat{\w \phi'}$.
\end{cor}
\begin{proof} For $\lambda \in \C$ we have $\phi_{\lambda}: x \mapsto x, \pa \mapsto \pa +
\lambda$. Let $\phi = \phi_1$. Let $\w \phi: x \mapsto x, \pa
\mapsto \pa + \f$, let $\w \phi_{1 \over 2}: x \mapsto x, \pa
\mapsto \pa + {1 \over 2}\cdot \f$  and let $\w \phi' = \w \phi_{1
\over 2} \circ \w \phi_{1 \over 2}$. Thus $\w \phi$ and $\w \phi'$
are lifts of $\phi$.

By proposition \ref{properties of hat} \textbf{d)} and proposition
\ref{evidence at 1}  we have that $\widehat{\widetilde{\phi}_{1
\over 2} \circ \widetilde{\phi}_{1 \over 2}} =
\widehat{\widetilde{\phi}}_{1 \over 2} \circ
\widehat{\widetilde{\phi}}_{1 \over 2} = \Id_{\BGG_2}$ while
$\widehat{\widetilde{\phi}}_{1}$ is given by $r \mapsto r+1$ and
$s \mapsto s$.
\end{proof}

If we consider the subgroup $B$ of $\End(\At)$ generated by
elements of the form
$$
x \mapsto x, \pa \mapsto \pa + P(x)\f \hbox{ and } x \mapsto x +
Q(\pa)\f, \pa \mapsto \pa
$$
for $P(x) \in \C[x]$ and $Q(\pa) \in \C[\pa]$, then an extension
of the computation of proposition \ref{evidence at 1} can be used
to show that $\,\widehat{ \ } (B \cap \End(\At)') \subseteq
\Aut(P_n(\Z))$. This is an indication (though not a proof) that
$\,\widehat{ \ } (\End(\At)') \subseteq \Aut(P_n(\Z))$. This and
corollary \ref{image at 1} are the evidence we have for conjecture
\ref{hat is surjective}.

\begin{remark} Note that the assumption that the limit in the proof of the proposition was taken over primes
$l$ was essential; e.g., one can show that the constant term of
$\lim_{l \to \infty, l \in 2\N} \Theta^{-1}_l((\pa +x\f)^l_{q_l})$
diverges (while this constant term is zero for all odd values of
$l$).
\end{remark}

\begin{remark}\label{Frobenius on C} If the map $a \mapsto a^l$ on $\C$ would have been a field automorphism with inverse
$\operatorname{Fr}^{-1}_{l}$ we could instead have defined
$\widehat{\phi}$ as follows: First, $\operatorname{Fr}^{-1}_{l}$
induces a ring automorphism of $\BGG_{2n}$ by acting on the
coefficients of a polynomial. Then we could have put
$$\widehat{\phi}(P)= \lim_{l \to \infty} \Theta^{-1}_l \circ \phi
\circ \Theta_l \circ \operatorname{Fr}^{-1}_{l}(P), \; P \in
\BGG.$$ Then the computation in proposition \ref{evidence at 1}
would have given
$$\widehat{\widetilde{\phi}}_{\lambda x^m}: s \mapsto s, r \mapsto r + \lambda s^m
$$
for \emph{any} $\lambda \in \C$ (with respect to the lift
$\widetilde{\phi}_{\lambda x^m}$ defined in the proof of the
proposition) which seems to be precisely what we would have
wanted. Also the composition $\;\widehat{ \ } \circ \w { \ }\;$
would probably have been independent of the choice of $\;\w { \
}\;$ in this case and one could expect that all elements of
$\End(\Atn)$ converge and that $\End(\Atn) \to \Aut(P_n)$ is
surjective, (in particular we wouldn't need to discuss
$\Z$-points).

The reduction modulo a prime method has an advantage here due to
the existence of a Frobenius map in each prime characteristic; for
the record, we have not been able to see how a non-standard
Frobenius on $\C$ could help to give a better definition of
$\;\widehat{ \ }\;$.
\end{remark}

\end{document}